\documentclass[10pt]{article}
\usepackage[a4paper, top=2cm, bottom=2cm, left=.3cm, right=.3cm]{geometry}
\usepackage{enumitem} 
\usepackage[nottoc,notlot,notlof]{tocbibind}
\usepackage{amsmath,amssymb,amsfonts,amsthm}
\usepackage{latexsym}
\usepackage{tikz}
\usepackage[normalem]{ulem}
\usetikzlibrary{automata,positioning}
\usetikzlibrary{chains,fit,shapes}
\usetikzlibrary{calc}
\usetikzlibrary{arrows}
\usepackage{float}
\usetikzlibrary{positioning,calc}
\usetikzlibrary{graphs}
\usetikzlibrary{graphs.standard}
\usetikzlibrary{arrows,decorations.markings}
\usepackage{tikz}
\usetikzlibrary{decorations.pathreplacing}
\usepackage{multicol}
\usepackage{xcolor}
\usepackage{hyperref}
\usepackage{abstract}
\usepackage{url}
\usepackage{ragged2e}
\usetikzlibrary{positioning}
\newtheorem{theorem}{Theorem}
\newtheorem{remark}{Remark}

\newtheorem{lemma}{Lemma}[section]

\theoremstyle{definition}

\newtheorem{observation}{Observation}[section]

\numberwithin{equation}{section}
\date{}
\title{Characterization of Word-Representable Near-Triangulations}

\author{\hspace{1cm} Suchanda Roy and Ramesh Hariharasubramanian \\ 
{{\footnotesize r.suchanda@iitg.ac.in},\ {\footnotesize  ramesh\_h@iitg.ac.in}}\\{\footnotesize Department of Mathematics, Indian Institute of Technology Guwahati, Guwahati, Assam 781039, India}}

\begin{document}
	\maketitle

	\begin{abstract}
 

A graph \(G=(V,E)\) is said to be \emph{word-representable} if there exists a word \(w\) over the alphabet \(V\) such that two distinct letters \(x,y\in V\) alternate in \(w\) if and only if \(xy\in E\). Word-representable graphs form a well-studied graph class with connections to graph orientations, combinatorics on words, and graph coloring.

A \emph{near-triangulation} is a planar graph in which every face except the outer face is a triangle. Several subclasses of near-triangulations have previously been investigated in the context of word-representability, including polyomino triangulations, triangulations of rectangular polyominoes with a single domino tile, \(K_4\)-free near-triangulations, face subdivisions of triangular grid graphs, triangulations of grid-covered cylinder graphs, and chordal near-triangulations.

In this paper, we obtain a complete characterization of word-representable near-triangulations in terms of forbidden induced subgraphs. Our result unifies and extends the previously known characterizations for the above subclasses, while also correcting inaccuracies appearing in earlier works.
sentable graph.\end{abstract}

	\maketitle
	\pagestyle{myheadings}
\section{Introduction}
The theory of \emph{word-representable graphs} has developed into a rich and active area of research with strong connections to algebra, graph theory, combinatorics on words, formal languages, and scheduling problems. The concept originated from the work of \textit{Kitaev} and \textit{Seif} on the celebrated \emph{Perkins semigroup}~\cite{kitaev2008word}, and was later formally introduced by \textit{Kitaev} and \textit{Pyatkin} in~\cite{kitaev2008representable}. A graph \(G=(V,E)\) is said to be \emph{word-representable} if there exists a word \(w\) over the alphabet \(V\) such that two distinct letters \(x,y\in V\) alternate in \(w\) if and only if \(xy\in E\). This graph class, which generalizes important classes such as circle graphs, comparability graphs, and 3-colorable graphs, has been extensively studied in monographs~\cite{kitaev2017comprehensive,kitaev2015words}. Despite extensive progress in the area, recognizing whether a graph is word-representable remains an NP-complete problem.

An orientation of a graph is \emph{transitive} if the presence of \( u \rightarrow v \) and \( v \rightarrow z \) implies \( u \rightarrow z \). Graphs that admit such an orientation are called comparability graphs.

An orientation of a graph is said to be \textit{semi-transitive} if it is acyclic and satisfies the following condition: for any directed path \( u_1 \rightarrow u_2 \rightarrow \cdots \rightarrow u_t \) with \( t \geq 2 \), either there is no edge from \( u_1 \) to \( u_t \), or the subgraph induced by the vertices \( u_1, \ldots, u_t \) forms a transitive tournament, i.e, all possible edges \( u_i \rightarrow u_j \) exist for every \( i < j \). An undirected graph is said to be semi-transitive if it can be oriented in such a way that the resulting digraph is semi-transitive. This notion of semi-transitive orientation was introduced in~\cite{HALLDORSSON2016164} as a tool for characterizing \textit{word-representable graphs} (Theorem \ref{ch}). 

Word-representability of planar graphs remains one of the most challenging problems in this area. Since all triangle-free planar graphs and planar graphs containing at most three triangles are known to be word-representable (Theorems~\ref{1}--\ref{3}), recent investigations have focused primarily on planar graph classes rich in triangles. In this direction, several subclasses have been studied, including polyomino triangulations~\cite{akrobotu2015word}, triangulations of rectangular polyominoes with a single domino~\cite{glen2015word}, \(K_4\)-free near-triangulations~\cite{glen2016colourability}, face subdivisions of triangular grid graphs~\cite{chen2016word}, triangulations of grid-covered cylinder graphs~\cite{chen2016word1}, and chordal near-triangulations~\cite{hauser2025word}.

Most of these graph classes are subclasses of \emph{near-triangulations}, that is, planar graphs in which every face except the outer face is a triangle. In this paper, we obtain a complete characterization of word-representable near-triangulations in terms of forbidden induced subgraphs, thereby unifying and generalizing all previously known characterizations for the subclasses mentioned above. In addition, we correct two inaccurate results appearing in earlier works. From another perspective, our characterization also yields a collection of forbidden induced subgraphs whose exclusion guarantees the closure of word-representability under \(W\)-components introduced in~\cite{salam2022chordal}.

The remainder of the paper is organized as follows. Section~\ref{preli} reviews the necessary definitions and known results regarding word-representability and near-triangulations. Section~\ref{result} contains our main results, and Section~\ref{con} concludes the paper with some directions for future research.
\section{Preliminaries}\label{preli}
One of the key developments in the study of word-representable graphs is their characterization in terms of semi-transitive orientations, which forms the basis for many subsequent results. We begin this section by stating the fundamental theorem below. For a comprehensive account of the theory and related results, see~\cite{akrobotu2015word,glen2015word,glen2016colourability,chen2016word,chen2016word1,hauser2025word}.
\begin{theorem}\cite{HALLDORSSON2016164}\label{ch} A graph is word-representable if and only if it admits a semi-transitive orientation.
\end{theorem}

As a corollary of this, \emph{Halldórsson et al.} proved the following theorem.
\begin{theorem}\label{1}\cite{HALLDORSSON2016164}
Every $3$-colorable graph is word-representable.
\end{theorem}
The above theorem connects word-representability with graph coloring. Together with well-known coloring results for planar graphs, it yields several important classes of word-representable planar graphs, as stated in the following theorems.

\begin{theorem}[Grötzsch]\label{2}\cite{thomassen2003short}
Every triangle-free planar graph is $3$-colorable.
\end{theorem}
\begin{theorem}\label{3}\cite{aksionovcontinuation}
    Every planar graph with at most $3$-triangles is $3$-colorable.
\end{theorem}

Together, these results show that planar graphs with few triangles pose no obstacle to word-representability. The challenge therefore, lies in understanding planar graphs rich in triangles. In particular, several subclasses of near-triangulations and related planar graph classes have previously been studied from the perspective of word-representability, leading to a variety of partial characterizations and structural results.
A recurring phenomenon in many of these works is the strong connection between word-representability, \(3\)-colorability, and forbidden induced subgraph characterizations.

The following theorems establish this relationship for several subclasses arising from polyomino triangulations and $K_4$- free near-triangulations.
\begin{theorem}\cite{akrobotu2015word}
A triangulation of a convex polyomino is word-representable if and only if it is \(3\)-colorable.
\end{theorem}
\begin{theorem}\cite{glen2015word}
A triangulation of a rectangular polyomino with a single domino tile is word-representable if and only if it is \(3\)-colorable.
\end{theorem}

\begin{theorem}\cite{glen2016colourability}\label{4}
A \(K_4\)-free triangulation of a polyomino with \(n\)-omino tiles and without internal holes is word-representable if and only if it is \(3\)-colorable.
\end{theorem}

\begin{theorem}\cite{glen2016colourability}
A \(K_4\)-free near-triangulation is word-representable if and only if it is \(3\)-colorable.
\end{theorem}

Beyond the above mentioned classes, several results have also been obtained for triangulations and subdivisions of grid-based graph classes.

\begin{theorem}\cite{akrobotu2015word}
A triangulation of a grid graph is word-representable if and only if it is $3$-colorable.

\end{theorem}

\begin{theorem}\cite{chen2016word}
A face subdivision of a triangular grid graph \(G\) is word-representable if and only if it contains no induced subgraph isomorphic to \(A\) (see Figure \ref{fig1}); equivalently, \(G\) has no subdivided interior cell.
\end{theorem}

Related characterizations have also been established for triangulations of grid-covered cylinder graphs (GCCGs).

\begin{theorem}\cite{chen2016word1}
A triangulation of a GCCG with more than three sectors is word-representable if and only if it contains no induced subgraph isomorphic to \(W_5\) or \(W_7\).
\end{theorem}

\begin{theorem}\cite{chen2016word1}
A triangulation of a GCCG with three sectors is word-representable if and only if it contains no graph from Figure~4.11 in \cite{chen2016word1} as an induced subgraph.
\end{theorem}
More recently, word-representability of chordal near-triangulations is characterized in terms of forbidden induced subgraphs in~\cite{hauser2025word}. This result is particularly significant in the present context, since chordal near-triangulations form a broad subclass of near-triangulations.
\begin{theorem}\cite{hauser2025word}\label{app}
    An Apollonian network is word-representable if and only if it is $\mathcal{F}_1$-free and locally comparability.
\end{theorem}

\begin{theorem}\cite{hauser2025word}
A chordal near-triangulation is word-representable if and only if it is locally comparability and $\{A, B, B', \mathcal{F}_{1}, \mathcal{F}_{2}\}$-free.
\end{theorem}
Together, these results suggest that word-representability of near-triangulation-type graph classes depends on strong structural properties, motivating the search for a unified characterization of general near-triangulations.

\section{Characterization of Word-Representable Near-Triangulations}\label{result}
In this work, we establish a complete characterization of word-representable near-triangulations. In the context of near-triangulations \cite{salam2022chordal}, Salam et al. introduced a decomposition of near-triangulations into W-components (i.e, 2-connected components free of both edge separators and triangle separators). However, from the viewpoint of word-representability,  this decomposition has a limitation: the class of word-representable graphs is not closed under gluing two word-representable graphs along a clique (clique separators in case of decomposition into $W$-components). Consequently, a structural characterization of word-representable near-triangulations is equivalent to identifying the induced subgraphs whose exclusion guarantees the closure property of word-representable W-components of a near-triangulation under clique glueing.

As part of this work, we first clarify two inaccuracies in earlier papers~\cite{glen2016colourability,kitaev2024human}. In particular, we rectify a result from~\cite{glen2016colourability} that \emph{a $K_4$-free near-triangulation is word-representable if and only if it is perfect}. We give an explicit counterexample (\emph{Figure \ref{fig2}}) which is $K_4$-free and non-perfect but still $3$-colorable as well as word-representable. Also, we provide the correct characterization for $K_4$-free word-representable near-triangulations as follows.
\begin{figure}[H]
\centering

\begin{minipage}{0.42\textwidth}
\centering
\begin{tikzpicture}[scale=1.7]

\node[circle,fill=black,inner sep=1.5pt,label=below:$1$] (A) at (0,0) {};
\node[circle,fill=black,inner sep=1.5pt,label=below:$2$] (B) at (2,0) {};
\node[circle,fill=black,inner sep=1.5pt,label=above:$3$] (C) at (1,1.732) {};

\node[circle,fill=black,inner sep=1.5pt,label=right:$4$] (D) at (1,0.58) {};

\draw (A)--(B)--(C)--cycle;

\draw (A)--(D);
\draw (B)--(D);
\draw (C)--(D);
\draw (A)--(C);

\end{tikzpicture}

\caption{$K_4$ as a subdivision of a triangular face}
\label{k4}
\end{minipage}
\hfill
\begin{minipage}{0.52\textwidth}
\centering

\begin{tikzpicture}[scale=1.2, every node/.style={inner sep=2pt},
font=\fontsize{1}{7}\tiny]

\node[circle, fill=red] (1) at (.5,1) {};
\node[circle, fill=blue] (2) at (1.5,1) {};
\node[circle, fill=green] (3) at (2.5,1) {};
\node[circle, fill=blue] (4) at (0,0) {};

\node[circle, fill=green] (5) at (1,0) {};
\node[circle, fill=red] (6) at (2,0) {};
\node[circle, fill=blue] (7) at (3,0) {};
\node[circle, fill=red] (8) at (.5,-1) {};
\node[circle, fill=blue] (9) at (1.5,-1) {};
\node[circle, fill=green] (10) at (2.5,-1) {};
\node[circle, fill=green] (11) at (1,-2) {};
\node[circle, fill=red] (12) at (2,-2) {};

\node[above=2pt] at (1) {$1$};
\node[above=2pt] at (2) {$2$};
\node[above=2pt] at (3) {$3$};
\node[above=3pt] at (4) {$4$};
\node[above=3pt] at (5) {$5$};
\node[above=3pt] at (6) {$6$};
\node[above=2pt] at (7) {$7$};
\node[left=2pt] at (8) {$8$};
\node[below=3pt] at (9) {$9$};
\node[right=2pt] at (10) {$10$};
\node[below=1pt] at (11) {$11$};
\node[below=2pt] at (12) {$12$};

\draw (1) -- (2) -- (3) -- (7) -- (10) -- (12) -- (11) -- (8) -- (4) -- (1);
\draw (1) -- (5) -- (9) -- (12);
\draw (3) -- (6) -- (9) -- (11);
\draw (4) -- (5) -- (6) -- (7);
\draw (2) -- (5) -- (8);
\draw (2) -- (6) -- (10);
\draw (8) -- (9) -- (10);

\end{tikzpicture}

\caption{A \(K_4\)-free \(3\)-colorable non-perfect near-triangulation}
\label{fig2}
\end{minipage}

\end{figure}

\begin{theorem}\label{k4free}
A $K_4$-free near-triangulation is word-representable if and only if it is odd-wheel-free i.e., contains no induced odd wheel $W_{2k+1}$ (see Figure~\ref{fig1}).
\end{theorem}
\begin{proof}
    In~\cite{diks2002new}, it was shown that \emph{a near-triangulation is \(3\)-colorable if and only if it is an even near-triangulation, that is, every inner vertex has even degree}. Combining this with Theorem~\ref{4}, it follows that \emph{a \(K_4\)-free near-triangulation is word-representable if and only if it is an even near-triangulation}. Therefore, it remains to prove that \emph{a near-triangulation is even if and only if it contains no induced odd wheel}.

$\Longrightarrow$
Suppose that \(G\) contains an induced odd wheel. Then the central vertex of the wheel has odd degree, implying that \(G\) is not an even near-triangulation.

$\Longleftarrow$
Suppose that \(G\) contains no induced odd wheel. Assume, for contradiction, that \(G\) is not an even near-triangulation. Then there exists an inner vertex \(v\) of odd degree. Since \(G\) is a near-triangulation, the neighbors of \(v\) induce a cycle, and together with \(v\), the closed neighborhood \(N[v]\) forms an odd wheel as a subgraph. Since \(G\) contains no induced odd wheel, this odd wheel must contain additional edges among the vertices of the rim cycle.

Let the vertices on the rim of the wheel be \(v_1,v_2,\dots,v_{2t+1}\) for some positive integer \(t\), and let \(c\) denote the center vertex. Since the wheel is not induced, there exist additional edges among the vertices of the rim cycle.

First observe that no additional edge can join two vertices at distance \(2\) on the rim. Indeed, if there exists an edge \(v_iv_{i+2}\), then the vertices \(v_i,v_{i+1},v_{i+2}\), together with \(c\), induce a \(K_4\), contradicting the assumption that \(G\) is \(K_4\)-free. Hence, every additional edge is of the form \(v_iv_j\), where the cyclic distance between \(v_i\) and \(v_j\) on the rim is at least \(3\).

For an additional edge \(v_iv_j\), define its \emph{covered segment} to be the shorter path on the rim cycle from \(v_i\) to \(v_j\), and define the \emph{cover number} to be the number of internal vertices on this path, that is, the number of vertices on the path excluding \(v_i\) and \(v_j\).

Now, we remove all vertices from the rim that lie strictly inside covered segments, provided they are not endpoints of any additional edge. Repeating this process, we obtain a smaller wheel whose rim consists only of endpoints of additional edges together with uncovered vertices.

If every additional edge has an even cover number, then each deletion removes an even number of vertices from the odd rim cycle. Hence, the resulting rim still has an odd length, and therefore the obtained induced subgraph is again an odd wheel, contradicting our assumption.

Therefore, there exists an odd number of additional edges having odd cover numbers. Choose such an edge \(v_iv_j\) with minimum cover number, and consider the induced subgraph on
\(
{v_i,v_{i+1},\dots,v_j,c}.
\)
By the minimality of the chosen cover number, there can be no additional edge inside the segment \(v_i,v_{i+1},\dots,v_j\). Consequently, the subgraph induced by  \(v_i,v_{i+1},\dots,v_j,c\) forms an odd wheel, again contradicting the assumption that \(G\) contains no induced odd wheel. Hence, \(G\) is an even near-triangulation.

\end{proof}

Additionally, we identify a new minimal non-word-representable graph on seven vertices ($A''$ in Figure~\ref{fig1}), which does not appear among all 25 minimal non-word-representable graphs on seven vertices listed by Kitaev et al. in \cite{kitaev2024human}.

As a consequence of Theorem~\ref{k4free}, \(K_4\)-free word-representable near-triangulations are completely characterized in terms of forbidden induced subgraphs. Thus, it remains to characterize near-triangulations containing \(K_4\). To understand the structure of such graphs, we first make the following observation.

\begin{observation}
 From the structure of \(K_4\) (\emph{see Figure~\ref{k4}}), it follows that a near-triangulation containing \(K_4\) can be viewed as being obtained from a $K_4$-free near-triangulation, by repeatedly subdividing triangular faces through the insertion of a new vertex inside a triangle and joining it to all three vertices of that triangle. 

 In~\cite{hauser2025word}, the chordal property imposes a strong structural restriction on near-triangulations. In particular, chordal near-triangulations whose outer face has length three essentially coincide with Apollonian networks, while those with an outer face of length four can be viewed as two distinct Apollonian networks attached along an edge. However, such a clean recursive structure is no longer available for general near-triangulations.
\end{observation}

As a first step toward understanding general near-triangulations, we consider the case where the outer face is a triangle. We say that a triangle is \emph{subdivided} if a new vertex is inserted inside the triangle and joined to all three vertices, as illustrated in Figure~\ref{k4}. The resulting inner triangles are then regarded as triangles obtained from face subdivisions.

From this viewpoint, such graphs can be broadly divided into the following three types.

\begin{itemize}
\item[\textbf{Type 1.}] Graphs in which no triangle arises from a face subdivision (see Figure \ref{type1}).
\item[\textbf{Type 2.}] Graphs in which every triangle arises from repeated face subdivisions (see Figure \ref{type2}).

\item[\textbf{Type 3.}] Graphs containing triangles of both kinds (see Figure \ref{type3}).

\end{itemize}

We characterize these three types of graph classes in the following lemmas.
\begin{lemma}
Let \(G\) be a Type--1 near-triangulation whose outer face has length \(3\). Then \(G\) is word-representable if and only if it is odd wheel \(W_{2k+1}\)-free.
\end{lemma}
\begin{proof}
By definition, a Type--1 near-triangulation contains no triangle arising from a face subdivision. Consequently, \(G\) cannot contain \(K_4\) as a subgraph, since the structure of \(K_4\) (see Figure~\ref{k4}) necessarily contains a subdivided triangular face. Hence, \(G\) is \(K_4\)-free.

Therefore, the result follows directly from Theorem~\ref{k4free}, which states that a \(K_4\)-free near-triangulation is word-representable if and only if it is odd wheel \(W_{2k+1}\)-free.
\end{proof}

\begin{figure}[H]
\centering

\begin{minipage}{0.3\textwidth}
\centering
\begin{tikzpicture}[scale=0.5]

\coordinate (A) at (0,4);
\coordinate (B) at (-4,0);
\coordinate (C) at (4,0);


\coordinate (F) at (-0.3,1.6);

\coordinate (I) at (0.6,1.5);
\coordinate (J) at (0,0.8);

\draw (A)--(B)--(C)--cycle;


\draw (A)--(F);
\draw (A)--(I);
\draw (B)--(F);
\draw (B)--(J);

\draw (F)--(J);

\draw (C)--(J);
\draw (C)--(I);


\draw (F)--(I);

\draw (I)--(J);

\foreach \p in {A,B,C,F,I,J}
\node[circle,fill=black,inner sep=1.8pt] at (\p) {};


\end{tikzpicture}

\caption{Type-1}
\label{type1}
\end{minipage}
\hfill
\begin{minipage}{0.3\textwidth}
\centering

\begin{tikzpicture}[scale=0.5]

\coordinate (A) at (0,4);
\coordinate (B) at (-4,0);
\coordinate (C) at (4,0);

\coordinate (D1) at (0,2.3);
\coordinate (D2) at (0,1.3);
\coordinate (D3) at (0,0.5);

\coordinate (R1) at (1.2,2.0);
\coordinate (R2) at (1.3,1.4);

\draw (A)--(B)--(C)--cycle;

\draw (A)--(D1)--(D2)--(D3);

\draw (B)--(D1);
\draw (B)--(D2);
\draw (B)--(D3);

\draw (C)--(D1);
\draw (C)--(D2);
\draw (C)--(D3);

\draw (D1)--(R1)--(D1);
\draw (R1)--(C);
\draw (R1)--(A);
\draw (B)--(C);


\foreach \p in {A,B,C,D1,D2,D3,R1}
\node[circle,fill=black,inner sep=1.8pt] at (\p) {};


\end{tikzpicture}

\caption{Type-2}
\label{type2}
\end{minipage}
\hfill
\begin{minipage}{0.3\textwidth}
\centering

\begin{tikzpicture}[scale=0.5]

\coordinate (A) at (0,4);
\coordinate (B) at (-4,0);
\coordinate (C) at (4,0);
\coordinate (A1) at (-.4,2.3);
\coordinate (B1) at (-1.2,2.2);
\coordinate (C1) at (-1,1.5);

\coordinate (D2) at (0,1.3);
\coordinate (D3) at (0,0.5);

\coordinate (R1) at (1,2.0);
\coordinate (R2) at (1.3,1.4);

\draw (A)--(B)--(C)--cycle;
\draw (A1)--(B1)--(C1)--cycle;

\draw (A)--(D2)--(D3);


\draw (B)--(D2);
\draw (B)--(D3);
\draw (B)--(B1);
\draw (B)--(C1);
\draw (A)--(B1);
\draw (A)--(A1);

\draw (C)--(D2);
\draw (C)--(D3);
\draw (D2)--(C1);
\draw (D2)--(A1);
\draw (D2)--(R1);
\draw (R1)--(C);
\draw (R1)--(A);
\draw (B)--(C);


\foreach \p in {A,B,C,A1,B1,C1,D2,D3,R1}
\node[circle,fill=black,inner sep=1.8pt] at (\p) {};


\end{tikzpicture}

\caption{Type-3}
\label{type3}
\end{minipage}

\end{figure}
\begin{remark}
    Any Type--2 near-triangulation with an outer face of length \(3\) is precisely an Apollonian network. Since the outer face is a triangle and every triangle in the graph arises from a face subdivision, the graph is obtained by repeatedly subdividing triangular faces through the insertion of a new vertex adjacent to all three vertices of the face. This is exactly the recursive construction of Apollonian networks. Therefore, the characterization of Type--2 near-triangulations with an outer face length \(3\) follows directly from Theorem~\ref{app}.

\end{remark}

\begin{lemma}\label{l3}
  Let $G$ be a type-3 near-triangulation with an outer face length 3. Then $G$ is not word-representable.  
\end{lemma}
\begin{proof}
Since all faces are triangles and both kinds of triangles occur—those arising from face subdivisions and those not arising from face subdivisions—there exists at least one triangle such that none of the triangles inside it arise from face subdivisions (see Figure~\ref{col}). Considering a planar embedding of the graph, we may assume that all faces arising from subdivisions lie inside this triangle, while all faces not arising from subdivisions lie outside it (see Figure~\ref{embedd}). As shown in Figure \ref{embedd}, depending on the presence of the dotted edges, the subgraph, induced by the vertices $a,b,c,d,g,h$ and $i$ $\in \{A,A',A'',A'''\}$. Hence, any graph considered under this lemma is non-word-representable.
\begin{figure}[H]
\centering
\begin{minipage}{0.45\textwidth}
\centering

\begin{tikzpicture}[scale=0.65]

\coordinate (A) at (0,4);
\coordinate (B) at (-4,0);
\coordinate (C) at (4,0);
\coordinate (A1) at (-.4,2.3);
\coordinate (B1) at (-1.2,2.2);
\coordinate (C1) at (-1,1.5);

\coordinate (D2) at (0,1.3);
\coordinate (D3) at (0,0.5);

\coordinate (R1) at (1,2.0);
\coordinate (R2) at (1.3,1.4);
\draw[thick,dotted] (A1)--(B1)--(C1)--cycle;

\draw(B)--(C);
\draw(A)--(C);
\draw[thick,red] (A)--(B);
\draw[thick,red] (A)--(D2);
\draw[thick,red] (B)--(D2);
\draw (D2)--(D3);

\draw[red] (B)--(D2);
\draw (B)--(D3);
\draw (B)--(B1);
\draw (B)--(C1);
\draw (A)--(B1);
\draw (A)--(A1);

\draw (C)--(D2);
\draw (C)--(D3);
\draw (D2)--(C1);
\draw (D2)--(A1);
\draw (D2)--(R1);
\draw (R1)--(C);
\draw (R1)--(A);
\draw (B)--(C);


\foreach \p in {A,B,C,A1,B1,C1,D2,D3,R1}
\node[circle,fill=black,inner sep=1.8pt] at (\p) {};

\node[above] at (A) {\tiny$a$};
\node[left] at (B) {\tiny$b$};
\node[right] at (C) {\tiny$c$};
\node[right] at (A1) {\tiny$g$};
\node[left] at (B1) {\tiny$h$};
\node[left] at (C1) {\tiny$i$};
\node[right] at (D2) {\tiny$d$};
\node[below] at (D3) {\tiny$f$};
\node[right] at (R1) {\tiny$e$};

\end{tikzpicture}

\caption{$\triangle abd$ contains no face arising from face subdivisions. The dotted edges may or may not be present.}
\label{col}
\end{minipage}
\hfill
\begin{minipage}{0.45\textwidth}
\centering

\begin{tikzpicture}[scale=0.65]

\coordinate (A) at (0,3);
\coordinate (B) at (-1.8,0);
\coordinate (C) at (2,0);
\coordinate (A1) at (3,2.3);
\coordinate (B1) at (-2.5,2.2);
\coordinate (C1) at (0,-1);

\coordinate (D2) at (0,1.3);
\coordinate (D3) at (0,0.5);

\coordinate (R1) at (.3,1.8);
\coordinate (R2) at (1.3,1.4);

\draw (A)--(B)--(C)--cycle;

\draw (A)--(D2)--(D3);

\draw (C)--(C1);
\draw (C)--(A1);
\draw (B)--(D2);
\draw (B)--(D3);
\draw (B)--(B1);
\draw (B)--(C1);
\draw (A)--(B1);
\draw (A)--(A1);
\draw[thick,red] (A)--(B);
\draw[thick,red] (A)--(C);
\draw[thick,red] (B)--(C);

\draw (C)--(D2);
\draw (C)--(D3);
\draw (D2)--(R1);
\draw (R1)--(C);
\draw (R1)--(A);


\foreach \p in {A,B,C,A1,B1,C1,D2,D3,R1}
\node[circle,fill=black,inner sep=1.8pt] at (\p) {};

\node[above] at (A) {\tiny$a$};
\node[left] at (B) {\tiny$b$};
\node[right] at (C) {\tiny$d$};
\node[right] at (A1) {\tiny$g$};
\node[left] at (B1) {\tiny$h$};
\node[below] at (C1) {\tiny$i$};
\node[right] at (D2) {\tiny$c$};
\node[below] at (D3) {\tiny$f$};
\node[right] at (R1) {\tiny$e$};
\draw[thick,dotted,bend right=90] (A1) to (B1);
\draw[thick,dotted,bend right=100] (B1) to (C1);
\draw[thick,dotted,bend right=90] (C1) to (A1);

\end{tikzpicture}

\caption{A planar embedding in which all triangles arising from face subdivisions lie inside $\triangle abd$, while triangles not arising from face subdivisions, together with possible dotted edges, lie outside $\triangle abc$.}
\label{embedd}
\end{minipage}

\end{figure}
\end{proof}
\begin{remark}
   From Lemma~\ref{l3}, it follows that a word-representable near-triangulation may contain either triangles arising from face subdivisions or triangles not arising from face subdivisions, but not both simultaneously. Hence, a near-triangulation with an outer face of length at most \(3\) is word-representable if and only if it is locally comparability and \({\mathcal{F}_1,A,A',A'',A'''}\)-free. Note that every locally comparability graph is necessarily odd-wheel-free, since odd cycles are not comparability graphs and the neighborhood of the central vertex in an odd wheel induces an odd cycle.
\end{remark}
In the following main characterization theorem, we extend the characterization from near-triangulations with an outer face of length at most \(3\) to general near-triangulations with an arbitrary outer face length.
\begin{theorem}
Let $G$ be a near-triangulation. Then $G$ is word-representable if and only if it contains none of the graphs from Figure~\ref{fig1} as an induced subgraph.
\end{theorem}
\begin{proof}[Sketch of the proof:]
The proof proceeds by induction on the length of the outer face. By Remark~2, near-triangulations with an outer face of length \(3\) are completely characterized. This establishes the initial structural framework of the argument.

For the base cases of the induction, we first characterize near-triangulations whose outer face has length \(4\) and \(5\), showing that such graphs are word-representable if and only if they contain none of the graphs from Figure~\ref{fig1} as induced subgraphs.

We next set our induction hypothesis: every near-triangulation with an outer face length at most \(n\) is word-representable if and only if it contains none of the graphs from Figure~\ref{fig1} as an induced subgraph. Using the structural decomposition of near-triangulations along boundary triangles together with the forbidden induced subgraph conditions, we then prove that the characterization also holds for near-triangulations with an outer face length \(n+1\). This completes the induction.

\end{proof}
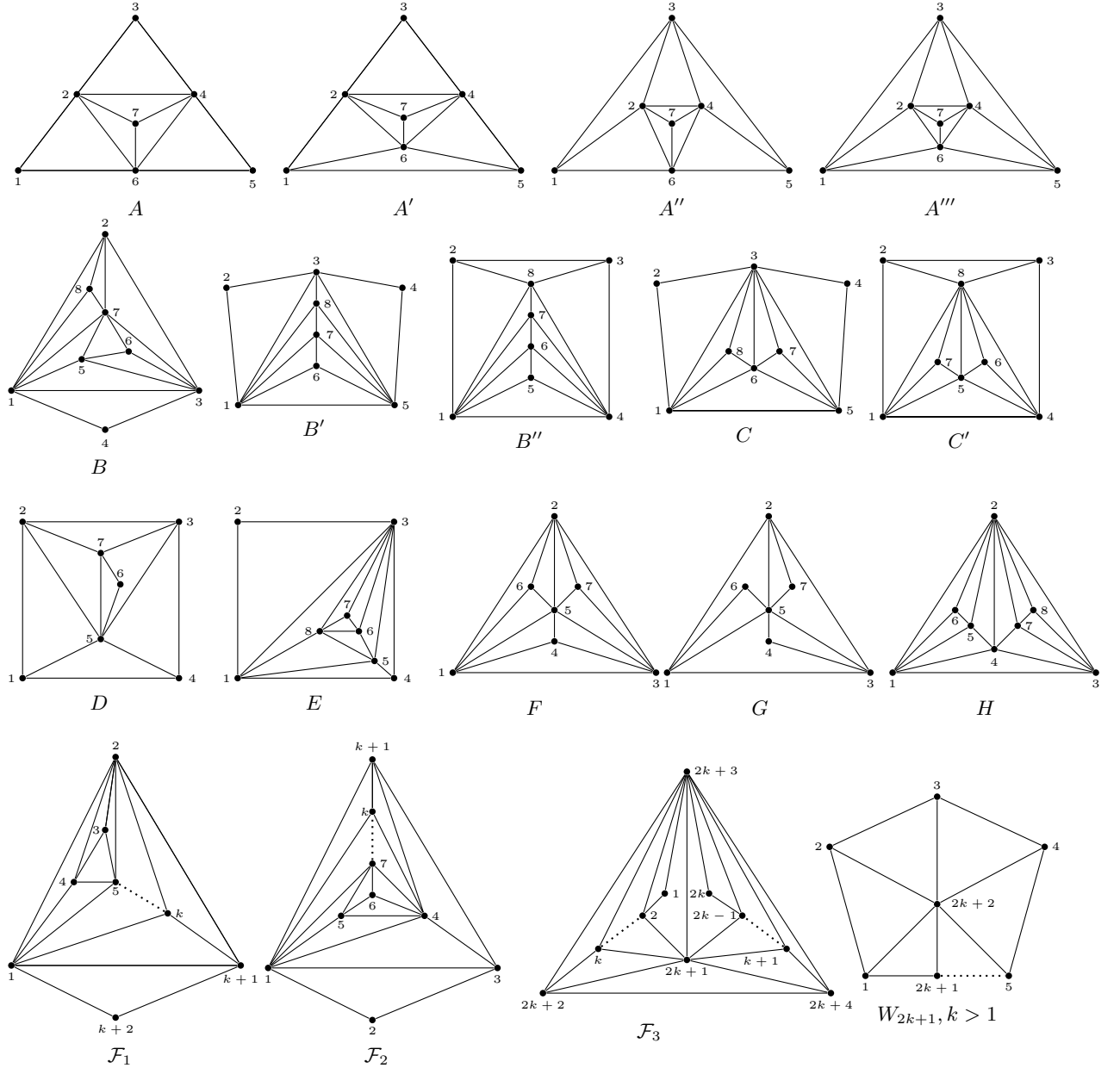
\begin{figure}[H]
\begin{tabular}{cccc}
\centering

\begin{minipage}{0.25\textwidth}
\centering
\begin{tikzpicture}[scale=.9, every node/.style={inner sep=1pt},
    font=\fontsize{1}{7}\tiny]
    \node[circle, fill] (T) at (0,2.6) {};
    \node[circle, fill] (L) at (-2,0) {};
    \node[circle, fill] (M) at (0,0) {};
    \node[circle, fill] (R) at (2,0) {};

    \node[circle, fill] (A) at (-1,1.3) {};
    \node[circle, fill] (B) at (1,1.3) {};
    \node[circle, fill] (C) at (0,.8) {};
    
    \node[above=2pt] at (T) {$3$};
    \node[below=2pt] at (L) {$1$};
    \node[below=2pt] at (M) {$6$};
    \node[below=3pt] at (R) {$5$};
    \node[left=2pt] at (A) {$2$};
    \node[right=1pt] at (B) {$4$};
    \node[above=2pt] at (C) {$7$};
    \draw (T) -- (B);
    \draw (R) -- (L);
    \draw (A) -- (L);
    \draw (T) -- (L);
    \draw (T) -- (R);
    \draw (A) -- (B);
    \draw (T) -- (A);
    \draw (R) -- (B);
    \draw (A) -- (C) -- (B);
    \draw (L) -- (M) -- (R);
    \draw (A) -- (M);
    \draw (C) -- (M);
    \draw (B) -- (M);
\end{tikzpicture}

\small\textbf{$A$}
\end{minipage}
\hfill
\begin{minipage}{0.25\textwidth}
\centering
\begin{tikzpicture}[scale=.9, every node/.style={inner sep=1pt},
    font=\fontsize{1}{7}\tiny]
    \node[circle, fill] (T) at (0,2.6) {};
    \node[circle, fill] (L) at (-2,0) {};
    \node[circle, fill] (M) at (0,.4) {};
    \node[circle, fill] (R) at (2,0) {};

    \node[circle, fill] (A) at (-1,1.3) {};
    \node[circle, fill] (B) at (1,1.3) {};
    \node[circle, fill] (C) at (0,0.9) {};
     \node[above=2pt] at (T) {$3$};
    \node[below=2pt] at (L) {$1$};
    \node[below=2pt] at (M) {$6$};
    \node[below=3pt] at (R) {$5$};
    \node[left=2pt] at (A) {$2$};
    \node[right=1pt] at (B) {$4$};
    \node[above=2pt] at (C) {$7$}; 
    \draw (T) -- (B);
    \draw (R) -- (L);
    \draw (A) -- (L);
    \draw (T) -- (L);
    \draw (T) -- (R);
    \draw (A) -- (B);
    \draw (T) -- (A);
    \draw (R) -- (B);
    \draw (A) -- (C) -- (B);
    \draw (L) -- (M) -- (R);
    \draw (A) -- (M);
    \draw (C) -- (M);
    \draw (B) -- (M);
\end{tikzpicture}

\small\textbf{$A'$}
\end{minipage}
\hfill
\begin{minipage}{0.25\textwidth}
\centering
\begin{tikzpicture}[scale=.9, every node/.style={inner sep=1pt},
    font=\fontsize{1}{7}\tiny]
    \node[circle, fill] (T) at (0,2.6) {};
    \node[circle, fill] (L) at (-2,0) {};
    \node[circle, fill] (M) at (0,0) {};
    \node[circle, fill] (R) at (2,0) {};

    \node[circle, fill] (A) at (-.5,1.1) {};
    \node[circle, fill] (B) at (.5,1.1) {};
    \node[circle, fill] (C) at (0,0.8) {};
\node[above=2pt] at (T) {$3$};
    \node[below=2pt] at (L) {$1$};
    \node[below=2pt] at (M) {$6$};
    \node[below=3pt] at (R) {$5$};
    \node[left=2pt] at (A) {$2$};
    \node[right=1pt] at (B) {$4$};
    \node[above=2pt] at (C) {$7$};
    \draw (T) -- (B);
    \draw (R) -- (L);
    \draw (A) -- (L);
    \draw (T) -- (L);
    \draw (T) -- (R);
    \draw (A) -- (B);
    \draw (T) -- (A);
    \draw (R) -- (B);
    \draw (A) -- (C) -- (B);
    \draw (A) -- (M);
    \draw (C) -- (M);
    \draw (B) -- (M);
\end{tikzpicture}

\small\textbf{$A''$}
\end{minipage}
\hfill
\begin{minipage}{0.25\textwidth}
\centering
\begin{tikzpicture}[scale=.9, every node/.style={inner sep=1pt},
    font=\fontsize{1}{7}\tiny]
    \node[circle, fill] (T) at (0,2.6) {};
    \node[circle, fill] (L) at (-2,0) {};
    \node[circle, fill] (M) at (0,.4) {};
    \node[circle, fill] (R) at (2,0) {};

    \node[circle, fill] (A) at (-.5,1.1) {};
    \node[circle, fill] (B) at (.5,1.1) {};
    \node[circle, fill] (C) at (0,0.8) {};
\node[above=2pt] at (T) {$3$};
    \node[below=2pt] at (L) {$1$};
    \node[below=2pt] at (M) {$6$};
    \node[below=3pt] at (R) {$5$};
    \node[left=2pt] at (A) {$2$};
    \node[right=1pt] at (B) {$4$};
    \node[above=2pt] at (C) {$7$};
    \draw (T) -- (B);
    \draw (R) -- (L);
    \draw (A) -- (L);
    \draw (T) -- (L);
    \draw (T) -- (R);
    \draw (A) -- (B);
    \draw (T) -- (A);
    \draw (R) -- (B);
    \draw (A) -- (C) -- (B);
    \draw (L) -- (M) -- (R);
    \draw (A) -- (M);
    \draw (C) -- (M);
    \draw (B) -- (M);
\end{tikzpicture}

\small\textbf{$A'''$}
\end{minipage}

\end{tabular}

   \begin{tabular}{ccccc}

\begin{minipage}{0.18\textwidth}
\centering
\begin{tikzpicture}[scale=1.2, every node/.style={inner sep=1pt},
    font=\fontsize{1}{7}\tiny]

    \node[circle, fill] (T)  at (0,2)    {};
    \node[circle, fill] (BL) at (-1.2,0) {};
    \node[circle, fill] (BR) at (1.2,0)  {};
    \node[circle, fill] (B)  at (0,-0.5) {};

    \node[circle, fill] (A1) at (-0.2,1.3) {};
    \node[circle, fill] (A3) at (0,1)      {};
    \node[circle, fill] (A4) at (-0.3,0.4) {};
    \node[circle, fill] (A5) at (0.3,0.5)  {};
     \node[above=2pt] at (T) {$2$};
    \node[below=2pt] at (BL) {$1$};
    \node[below=2pt] at (BR) {$3$};
    \node[below=3pt] at (B) {$4$};
    \node[left=2pt] at (A1) {$8$};
    \node[right=2pt] at (A3) {$7$};
    \node[below=2pt] at (A4) {$5$};
    \node[above=2pt] at (A5) {$6$};
    \draw (BL) -- (BR) -- (B) -- (BL);

    \draw (T) -- (A1);
    \draw (T) -- (A3);
    \draw (T) -- (BL);
    \draw (T) -- (BR);

    \draw (A1) -- (A3);
    \draw (A4) -- (A3);
    \draw (A4) -- (BR);
    \draw (A4) -- (BL);

    \draw (A5) -- (A3);
    \draw (A5) -- (BR);
    \draw (A5) -- (A4);

    \draw (BL) -- (A1);
    \draw (BL) -- (A3);
    \draw (BR) -- (A3);

\end{tikzpicture}

\small \textbf{$B$}
\end{minipage}
&
\begin{minipage}{0.18\textwidth}
\centering
\begin{tikzpicture}[scale=1.2, every node/.style={inner sep=1pt},
    font=\fontsize{1}{7}\tiny]
    \node[circle, fill] (T)  at (-0.15,1.5) {};
    \node[circle, fill] (A1) at (0,0)       {};
    \node[circle, fill] (A2) at (1,1.7)     {};
    \node[circle, fill] (A3) at (2,0)       {};
    \node[circle, fill] (TL) at (2.1,1.5)   {};
    \node[circle, fill] (L)  at (1,1.3)     {};
    \node[circle, fill] (R)  at (1,0.5)     {};
    \node[circle, fill] (R1) at (1,0.9)     {};
\node[above=2pt] at (T) {$2$};
    \node[right=2pt] at (TL) {$4$};
    \node[below=2pt] at (R) {$6$};
    \node[right=3pt] at (R1) {$7$};
    \node[left=2pt] at (A1) {$1$};
    \node[right=2pt] at (A3) {$5$};
    \node[above=2pt] at (A2) {$3$};
    \node[right=2pt] at (L) {$8$};
    \draw (T) -- (A1) -- (A2) -- (A3) -- (TL);
    \draw (T) -- (A2) -- (TL);

    \draw (A2) -- (L) -- (R);
    \draw (A1) -- (R) -- (A3);
    \draw (A1) -- (L) -- (A3);
    \draw (A1) -- (R1) -- (A3);

    \draw (A1) -- (A3);

\end{tikzpicture}

\small \textbf{$B'$}
\end{minipage}
&
\begin{minipage}{0.18\textwidth}
\centering

\begin{tikzpicture}[scale=1.2, every node/.style={inner sep=1pt},
    font=\fontsize{1}{7}\tiny]
    \node[circle, fill] (T)  at (0,2)   {};
    \node[circle, fill] (A1) at (0,0)   {};
    \node[circle, fill] (A2) at (1,1.7) {};
    \node[circle, fill] (A3) at (2,0)   {};
    \node[circle, fill] (TL) at (2,2)   {};
    \node[circle, fill] (L)  at (1,1.3) {};
    \node[circle, fill] (R)  at (1,0.5) {};
    \node[circle, fill] (R1) at (1,0.9) {};
\node[above=2pt] at (T) {$2$};
    \node[right=2pt] at (TL) {$3$};
    \node[below=2pt] at (R) {$5$};
    \node[right=3pt] at (R1) {$6$};
    \node[left=2pt] at (A1) {$1$};
    \node[right=2pt] at (A3) {$4$};
    \node[above=2pt] at (A2) {$8$};
    \node[right=2pt] at (L) {$7$};
    \draw (T) -- (A1) -- (A2) -- (A3) -- (TL) -- (T);
    \draw (T) -- (A2) -- (TL);

    \draw (A2) -- (L) -- (R);
    \draw (A1) -- (R) -- (A3);
    \draw (A1) -- (L) -- (A3);
    \draw (A1) -- (R1) -- (A3);

    \draw (A1) -- (A3);

\end{tikzpicture}

\small \textbf{$B''$}
\end{minipage}

&
\begin{minipage}{0.18\textwidth}
\centering
  \begin{tikzpicture}[scale=1.3, every node/.style={inner sep=1pt},
    font=\fontsize{1}{7}\tiny]
    \node[circle, fill] (T) at (-.15,1.5) {};        
    \node[circle, fill] (A1) at (0,0) {};       
    \node[circle, fill] (A2) at (1,1.7) {};       
    \node[circle, fill] (A3) at (2,0) {};       
    \node[circle, fill] (TL) at (2.1,1.5) {};        
     \node[circle, fill] (L) at (.7,.7) {};  
     \node[circle, fill] (R1) at (1.3,.7) {};  
    \node[circle, fill] (R) at (1,.5) {};  
    \node[above=2pt] at (T) {$2$};
    \node[right=2pt] at (TL) {$4$};
    \node[below=2pt] at (R) {$6$};
    \node[right=3pt] at (R1) {$7$};
    \node[left=2pt] at (A1) {$1$};
    \node[right=2pt] at (A3) {$5$};
    \node[above=2pt] at (A2) {$3$};
    \node[right=2pt] at (L) {$8$};

    \draw (T) -- (A1) -- (A2) -- (A3) -- (TL);
    \draw (T) -- (A2) -- (TL);
    \draw (A2)  -- (R);
    \draw (A1) -- (R) -- (A3);
    \draw (A1)  -- (A3);
    \draw (A1)  -- (A3);
    \draw (A1) -- (A3);
    \draw (A1) -- (A3);
    \draw (A1) -- (A3);
    \draw (A1) -- (L);
    \draw (L) -- (R);
    \draw (L) -- (A2);
    \draw (R) -- (R1);
    \draw (A2) -- (R1);
    \draw (A3) -- (R1);
    \end{tikzpicture} 
    \small \textbf{$C$} 
    \end{minipage}
 &
\begin{minipage}{0.18\textwidth}
\centering
     \begin{tikzpicture}[scale=1.2, every node/.style={inner sep=1pt},
    font=\fontsize{1}{7}\tiny]
    \node[circle, fill] (T) at (0,2) {};        
    \node[circle, fill] (A1) at (0,0) {};       
    \node[circle, fill] (A2) at (1,1.7) {};       
    \node[circle, fill] (A3) at (2,0) {};       
    \node[circle, fill] (TL) at (2,2) {};        
     \node[circle, fill] (L) at (.7,.7) {};  
     \node[circle, fill] (R1) at (1.3,.7) {};  
    \node[circle, fill] (R) at (1,.5) {};        
  \node[above=2pt] at (T) {$2$};
    \node[right=2pt] at (TL) {$3$};
    \node[below=2pt] at (R) {$5$};
    \node[right=3pt] at (R1) {$6$};
    \node[left=2pt] at (A1) {$1$};
    \node[right=2pt] at (A3) {$4$};
    \node[above=2pt] at (A2) {$8$};
    \node[right=2pt] at (L) {$7$};
    \draw (T) -- (A1) -- (A2) -- (A3) -- (TL);
    \draw (T) -- (A2) -- (TL);
    \draw (A2)  -- (R);
    \draw (A1) -- (R) -- (A3);
    \draw (A1)  -- (A3);
    \draw (A1)  -- (A3);
    \draw (A1) -- (A3);
    \draw (A1) -- (A3);
    \draw (A1) -- (A3);
    \draw (A1) -- (L);
    \draw (L) -- (R);
    \draw (L) -- (A2);
    \draw (R) -- (R1);
    \draw (A2) -- (R1);
    \draw (A3) -- (R1);
    \draw (T) -- (TL);
    \end{tikzpicture} 
    \small \textbf{$C'$}
    \end{minipage}

\end{tabular}

  \vspace{1em}    
 \begin{tabular}{ccccc}

\begin{minipage}{0.18\textwidth}
\centering
     \begin{tikzpicture}[scale=1.2, every node/.style={inner sep=1pt},
    font=\fontsize{1}{7}\tiny]
    \node[circle, fill] (UL) at (0,2) {};        
    \node[circle, fill] (BL) at (0,0) {};       
    \node[circle, fill] (UR) at (2,2) {};       
    \node[circle, fill] (BR) at (2,0) {};      
    \node[circle, fill] (C) at (1,.5) {};        
     \node[circle, fill] (UC) at (1,1.6) {};  
     \node[circle, fill] (UC1) at (1.25,1.2) {};  
     \node[above=2pt] at (UL) {$2$};
    \node[left=2pt] at (BL) {$1$};
    \node[right=2pt] at (UR) {$3$};
    \node[right=3pt] at (BR) {$4$};
    \node[left=2pt] at (C) {$5$};
    \node[above=2pt] at (UC) {$7$};
    \node[above=2pt] at (UC1) {$6$};

    \draw (UL) -- (BL) -- (BR) -- (UR) -- (UL);
    \draw (UL) -- (C);
    \draw (UR) -- (C);
    \draw (BL) -- (C);
    \draw (BR) -- (C);
     \draw (UL) -- (UC);
    \draw (UR) -- (UC);
    \draw (C) -- (UC);
     \draw (C) -- (UC1);
     \draw (UC) -- (UC1);
    
    \end{tikzpicture} 
    \small \textbf{$D$} 
    \end{minipage}
 &
   
\begin{minipage}{0.18\textwidth}
\centering
      \begin{tikzpicture}[scale=1.2, every node/.style={inner sep=1pt},
    font=\fontsize{1}{7}\tiny]
    \node[circle, fill] (UL) at (0,2) {};        
    \node[circle, fill] (BL) at (0,0) {};       
    \node[circle, fill] (UR) at (2,2) {};       
    \node[circle, fill] (BR) at (2,0) {}; 
    
    \node[circle, fill] (C) at (1.05,.6) {};        
    \node[circle, fill] (SC) at (1.75,.22) {};
    \node[circle, fill] (C1) at (1.4,.8) {};    
    \node[circle, fill] (C2) at (1.55,.6) {};
    \node[above=2pt] at (UL) {$2$};
    \node[left=2pt] at (BL) {$1$};
    \node[right=2pt] at (UR) {$3$};
    \node[right=3pt] at (BR) {$4$};
    \node[left=2pt] at (C) {$8$};
    \node[right=2pt] at (SC) {$5$};
    \node[above=2pt] at (C1) {$7$};
    \node[right=2pt] at (C2) {$6$};

    \draw (UL) -- (BL) -- (BR) -- (UR) -- (UL);
    \draw (UR) -- (C);
    \draw (BL) -- (C);
     \draw (SC) -- (C);
      \draw (UR) -- (SC);
    \draw (BL) -- (SC);
     \draw (SC) -- (BR);
    \draw (C1) -- (C);
    \draw (C1) -- (C2);
    \draw (C1) -- (UR);
    \draw (C2) -- (C);
    \draw (C2) -- (SC);
    \draw (C2) -- (UR);
     \draw (BL) -- (UR);
    \end{tikzpicture} 
    \small \textbf{$E$}
    \end{minipage} 
&
\begin{minipage}{0.19\textwidth}
\centering
        \begin{tikzpicture}[scale=1.2, every node/.style={inner sep=1pt},
    font=\fontsize{1}{7}\tiny]
    \node[circle, fill] (1) at (0,0) {};      
    \node[circle, fill] (2) at (2.6,0) {};       
    \node[circle, fill] (3) at (1.3,2) {};        
    \node[circle, fill] (4) at (1.3,.8) {};        

    \node[circle, fill] (5) at (1,1.1) {};     
    \node[circle, fill] (6) at (1.6,1.1) {};      
    \node[circle, fill] (7) at (1.3,.4) {};      
    \node[left=2pt] at (1) {$1$};
    \node[below=2pt] at (2) {$3$};
    \node[above=2pt] at (3) {$2$};
    \node[right=3pt] at (4) {$5$};
    \node[left=2pt] at (5) {$6$};
    \node[right=2pt] at (6) {$7$};
    \node[below=2pt] at (7) {$4$};
    
    \draw (1) -- (2);
    \draw (2) -- (3);
    \draw (3) -- (1);
     \draw (1) -- (4);
    \draw (2) -- (4);
    \draw (3) -- (4);
    \draw (5) -- (1);   
    \draw (5) -- (3); 
    \draw (5) -- (4);
\draw (6) -- (3);   
    \draw (6) -- (2); 
    \draw (6) -- (4);
    \draw (7) -- (1);   
    \draw (7) -- (2); 
    \draw (7) -- (4);

\end{tikzpicture} 
\small \textbf{$F$}
\end{minipage}
 &
\vspace{1em}

\begin{minipage}{0.19\textwidth}
\centering
\begin{tikzpicture}[scale=1.2, every node/.style={inner sep=1pt},
    font=\fontsize{1}{7}\tiny]
    \node[circle, fill] (1) at (0,0) {};      
    \node[circle, fill] (2) at (2.6,0) {};       
    \node[circle, fill] (3) at (1.3,2) {};        
    \node[circle, fill] (4) at (1.3,.8) {};        

    \node[circle, fill] (5) at (1,1.1) {};     
    \node[circle, fill] (6) at (1.6,1.1) {};      
    \node[circle, fill] (7) at (1.3,.4) {};      
    \node[below=2pt] at (1) {$1$};
    \node[below=2pt] at (2) {$3$};
    \node[above=2pt] at (3) {$2$};
    \node[right=3pt] at (4) {$5$};
    \node[left=2pt] at (5) {$6$};
    \node[right=2pt] at (6) {$7$};
    \node[below=2pt] at (7) {$4$};
    \draw (1) -- (2);
    \draw (2) -- (3);
    \draw (3) -- (1);
     \draw (1) -- (4);
    \draw (2) -- (4);
    \draw (3) -- (4);
    \draw (5) -- (1);   
     
    \draw (5) -- (4);
    \draw (6) -- (3);   
    
    \draw (6) -- (4);
    \draw (7) -- (2); 
    \draw (7) -- (4);

\end{tikzpicture} 
\small \textbf{$G$} 
\end{minipage}
 &

\begin{minipage}{0.19\textwidth}
\centering
\begin{tikzpicture}[scale=1.2, every node/.style={inner sep=1pt},
    font=\fontsize{1}{7}\tiny]
    \node[circle, fill] (1) at (0,0) {};      
    \node[circle, fill] (2) at (2.6,0) {};       
    \node[circle, fill] (3) at (1.3,2) {};        
    \node[circle, fill] (4) at (1.3,.3) {};        

    \node[circle, fill] (5) at (1,.6) {};     
    \node[circle, fill] (6) at (1.6,.6) {};      
    \node[circle, fill] (7) at (.8,.8) {};      
     \node[circle, fill] (8) at (1.8,.8) {};
         \node[below=2pt] at (1) {$1$};
    \node[below=2pt] at (2) {$3$};
    \node[above=2pt] at (3) {$2$};
    \node[below=3pt] at (4) {$4$};
    \node[below=2pt] at (5) {$5$};
    \node[right=2pt] at (6) {$7$};
    \node[below=2pt] at (7) {$6$};
    \node[right=2pt] at (8) {$8$};
    \draw (1) -- (2);
    \draw (2) -- (3);
    \draw (3) -- (1);
     \draw (1) -- (4);
    \draw (2) -- (4);
    \draw (3) -- (4);
    \draw (5) -- (1);   
    \draw (5) -- (3); 
    \draw (5) -- (4);
    \draw (6) -- (3);   
    \draw (6) -- (2); 
    \draw (6) -- (4);
    \draw (7) -- (1);   
    \draw (7) -- (3); 
    \draw (7) -- (5);
        \draw (8) -- (2);   
    \draw (8) -- (6); 
    \draw (8) -- (3);
    
\end{tikzpicture} 
\small \textbf{$H$} 
\end{minipage}

\end{tabular}
 \begin{tabular}{cccc}
  \begin{minipage}{0.22\textwidth}
\centering
\begin{tikzpicture}[scale=1.6, every node/.style={inner sep=1pt},
    font=\fontsize{1}{7}\tiny]

    \node[circle, fill] (1) at (1,2)   {};
    \node[circle, fill] (2) at (0,0)   {};
    \node[circle, fill] (3) at (1,-0.5){};
    \node[circle, fill] (4) at (2.2,0) {};

    \node[circle, fill] (5) at (1.5,0.5){};
    \node[circle, fill] (6) at (1,0.8)  {};
    \node[circle, fill] (7) at (0.6,0.8){};
    \node[circle, fill] (8) at (0.9,1.3){};
     \node[above=2pt] at (1) {$2$};
    
    \node[left=1pt] at (8) {$3$};
    \node[below=2pt] at (2) {$1$};
    \node[below=2pt] at (3) {$k+2$};
    \node[below=3pt] at (4) {$k+1$};
    \node[right=2pt] at (5) {$k$};
    \node[below=1pt] at (6) {$5$};
    \node[left=2pt] at (7) {$4$};
    
    \draw (1) -- (2) -- (3) -- (4) -- (1);

    \draw (1) -- (7) -- (2);
    \draw (1) -- (6) -- (2);
    \draw (1) -- (5) -- (2);
    \draw (1) -- (4) -- (2);

    \draw (1) -- (8) -- (7);
    \draw (1) -- (8) -- (6);

    \draw (7) -- (6);
    \draw (5) -- (4);
    \draw (2) -- (4);

    \draw[dotted, thick] (5) -- (6);

\end{tikzpicture}

\small \textbf{$\mathcal{F}_1$}
\end{minipage}
&

\begin{minipage}{0.22\textwidth}
\centering
\begin{tikzpicture}[scale=1.6, every node/.style={inner sep=1pt},
    font=\fontsize{1}{7}\tiny]

    \node[circle, fill] (1) at (1,2)   {};
    \node[circle, fill] (2) at (0,0)   {};
    \node[circle, fill] (3) at (1,-0.5){};
    \node[circle, fill] (4) at (2.2,0) {};

    \node[circle, fill] (5) at (1.5,0.5){};
    \node[circle, fill] (6) at (1,1.5)  {};
    \node[circle, fill] (7) at (1,1)    {};
    \node[circle, fill] (8) at (0.7,0.5){};
    \node[circle, fill] (9) at (1,0.7)  {};
    \node[above=2pt] at (1) {$k+1$};
    
    \node[below=1pt] at (8) {$5$};
    \node[below=2pt] at (2) {$1$};
    \node[below=2pt] at (3) {$2$};
    \node[below=2pt] at (4) {$3$};
    \node[right=2pt] at (5) {$4$};
    \node[left=.1pt] at (6) {$k$};
    \node[right=2pt] at (7) {$7$};
    \node[below=2pt] at (9) {$6$};

    \draw (1) -- (2) -- (3) -- (4) -- (1);

    \draw (1) -- (6) -- (2);
    \draw (1) -- (6) -- (5) -- (1);
    \draw (2) -- (7) -- (5) -- (4) -- (2);
    \draw (2) -- (8) -- (5) -- (2);

    \draw (9) -- (8);
    \draw (7) -- (9);
    \draw (5) -- (9);
    \draw (7) -- (8);

    \draw[dotted, thick] (7) -- (6);

\end{tikzpicture}

\small \textbf{$\mathcal{F}_2$}
\end{minipage}
&

\begin{minipage}{0.25\textwidth}
\centering
\begin{tikzpicture}[scale=1.7, every node/.style={inner sep=1pt},
    font=\fontsize{1}{7}\tiny]
    \node[circle, fill] (1) at (0,0) {};      
    \node[circle, fill] (2) at (2.6,0) {};       
    \node[circle, fill] (3) at (1.3,2) {};        
    \node[circle, fill] (4) at (1.3,.3) {};        

    \node[circle, fill] (5) at (1.1,.9) {};     
    \node[circle, fill] (6) at (1.5,.9) {};      
    \node[circle, fill] (7) at (.9,.7) {};      
     \node[circle, fill] (8) at (1.8,.7) {};
     \node[circle, fill] (9) at (.5,.4) {};      
     \node[circle, fill] (10) at (2.2,.4) {};

     \node[below=2pt] at (1) {$2k+2$};
    
    \node[left=1pt] at (8) {$2k-1$};
    \node[below=2pt] at (2) {$2k+4$};
    \node[right=2pt] at (3) {$2k+3$};
    \node[below=2pt] at (4) {$2k+1$};
    \node[right=2pt] at (5) {$1$};
    \node[left=.1pt] at (6) {$2k$};
    \node[right=2pt] at (7) {$2$};
    \node[below=2pt] at (9) {$k$};
    \node[below left=3pt] at (10) {$k+1$};
    \draw (1) -- (2);
    \draw (2) -- (3);
    \draw (3) -- (1);
     \draw (1) -- (4);
    \draw (2) -- (4);
    \draw (3) -- (4);

    \draw (9) -- (1);   
    \draw (9) -- (4); 
    \draw (9) -- (3);

    \draw (10) -- (2);   
     \draw (10) -- (3); 
     \draw (10) -- (4);

    \draw (5) -- (3); 
   \draw (6) -- (3);   
    \draw (7) -- (5);   
    \draw (7) -- (3); 
    \draw (7) -- (4);
        \draw (8) -- (4);   
    \draw (8) -- (6); 
    \draw (8) -- (3);
    \draw[dotted, thick] (9) -- (7);
    \draw[dotted, thick] (10) -- (8);
\end{tikzpicture}
\small \textbf{$\mathcal{F}_3$} 
\end{minipage}
&
\begin{minipage}{0.25\textwidth}
\centering
     \begin{tikzpicture}[scale=1.1, every node/.style={inner sep=1pt},
    font=\fontsize{1}{7}\tiny]

    \node[circle, fill] (1) at (0,0) {};
    
    \node[circle, fill] (8) at (1,0) {};
    \node[circle, fill] (2) at (2,0) {};
    \node[circle, fill] (3) at (2.5,1.8) {};
    \node[circle, fill] (4) at (1,2.5) {};
    \node[circle, fill] (5) at (-0.5,1.8) {};
    \node[circle, fill] (6) at (1,1) {};

    \node[below=2pt] at (1) {$1$};
    
    \node[below=2pt] at (8) {$2k+1$};
    \node[below=2pt] at (2) {$5$};
    \node[right=2pt] at (3) {$4$};
    \node[above=2pt] at (4) {$3$};
    \node[left=2pt] at (5) {$2$};
    \node[right=5pt] at (6) {$2k+2$};

    \draw (2) -- (3) -- (4) -- (5) -- (1);

    \draw (6) -- (1);
    \draw (6) -- (2);
    \draw (6) -- (3);
    \draw (6) -- (4);
    \draw (6) -- (5);

    
    \draw (6) -- (8);
    
    \draw (1) -- (8);

    \draw[dotted, thick] (8) -- (2);

\end{tikzpicture}
    \small \textbf{$W_{2k+1}, k>1$}
    \end{minipage}
 \end{tabular}
    \caption{Forbidden induced subgraphs for word-representable near-triangulations; the last column contains infinite families.}
    \label{fig1}
\end{figure}







\section{Concluding Remarks}\label{con}
In this work, we have provided a complete characterization of word-representable near-triangulations. From another perspective, this characterization can also be viewed as identifying the induced subgraphs whose exclusion guarantees that closure property of word-representable \(W\)-components under clique gluing. In general, the class of word-representable graphs is not closed under gluing along cliques, which makes such a characterization particularly significant.

This viewpoint may also be relevant in the broader problem of characterizing word-representable planar graphs. Any planar graph containing a \(K_4\) admits a natural triangular separator given by the outer triangle of the \(K_4\) through the planar embedding of \(K_4\). Moreover, since the maximum clique size in a planar graph is \(4\), every clique separator in a planar graph has size at most $3$. Consequently, planar graphs can often be decomposed into \(W\)-components along clique separators of size at most $3$. Understanding when word-representability is preserved under such decompositions could therefore provide a direction toward characterizing word-representable planar graphs.

\bibliographystyle{plain}
\bibliography{WR-NT}
\end{document}